\documentclass[11pt]{article}
\usepackage{fancyhdr,graphicx}
\usepackage{amsfonts}
\usepackage{amssymb}
\usepackage{amsthm}
\usepackage{newlfont}
\usepackage{amsmath}
\usepackage{multicol}
\usepackage{subfigure}

\newtheorem{theorem}{Theorem}[section]
\newtheorem{lemma}[theorem]{Lemma}
\newtheorem{remark}[theorem]{Remark}

\newtheorem{proposition}[theorem]{Proposition}
\newtheorem{definition}[theorem]{Definition}

\newtheorem{example}[theorem]{Example}
\graphicspath{{figures/}}

\usepackage{mathrsfs}
\usepackage{subfig}
\usepackage{picinpar}
\begin{document}
	
	\title
	{Twist polynomial as a weight system for set systems}
	
	\author{Qingying Deng\\
		{\small School of Mathematics and Computational Science, Xiangtan University, P. R. China}\\
		Xian'an Jin\\
		{\small School of Mathematical Sciences, Xiamen University, P. R. China}\\
		Qi Yan\footnote{Corresponding author.}\\
		{\small School of Mathematics and Statistics, Lanzhou University, P. R. China}\\
		{\small{\t Email: qingying@xtu.edu.cn (Q. Deng), xajin@xmu.edu.cn (X. Jin), yanq@lzu.edu.cn (Q. Yan)}}
	}
	\date{}
	
	\maketitle
	
	\begin{abstract}
		
		Recently, Chmutov proved that the partial-dual polynomial considered as a function on chord diagrams satisfies the four-term relation. Deng et al. then proved that this function on framed chord diagrams also satisfies the four-term relation, i.e., is a framed weight system. In this paper, we extend their results to the twist polynomial of a set system by proving that the twist polynomial on set systems satisfies the four-term relation.
		
	\end{abstract}
	
	$\mathbf{keywords:} $ Set system, four-term relation, handle sliding, exchange handle ends

	\section{Introduction}
	
	The concept of partial duality was introduced by Chmutov in \cite{Chmutov} as a generalization of the notion of geometric duality of a ribbon graph. Gross, Mansour and Tucker \cite{Gross} used this concept of partial duality to define the partial-dual (genus) polynomial, encoding the Euler genera of all partial duals of a ribbon graph.
	The concept of twist (width) polynomial of a delta-matroid was introduced in 2022 by Yan and Jin in \cite{Yan}, as a generalization of the notion of partial-dual (genus) polynomial $^{\partial}\varepsilon_{G}(z)$ of a ribbon graph $G$.
	
	Recently, Chmutov \cite{Chmutov2023} proved that the partial-dual polynomial considered as a function on chord diagrams satisfies the four-term relation. Deng et al. \cite{Deng} then proved that this function on framed chord diagrams also satisfies the four-term relation, i.e., is a framed weight system.	
	In this paper, we extend their results to the twist polynomial of a set system by proving that twist polynomial of set systems satisfies the four-term relation. We also discuss the application of the four-term relation to binary delta-matroids and ribbon graphs.
	
	\section{Preliminaries}
	
	A set system $D=(E,\mathcal{F})$ is a finite set $E$ together with a subset $\mathcal{F}$ of the set $2^{E}$ of all subsets in $E$. The set $E$ is called the \emph{ground set} of the set system, and elements of $\mathcal{F}$ are its \emph{feasible sets}. $D$ is \emph{proper} if $\mathcal{F}\neq \emptyset$. Below, we consider only proper set systems without explicitly indicating this.
	
	As introduced by Bouchet in \cite{AB1}, a \emph{delta-matroid} is a set system $D=(E, \mathcal{F})$ such that if $X, Y \in \mathcal{F}$ and $u\in X\Delta Y$, then there is $v\in X\Delta Y$ (possibly  $v=u$ ) such that $X\Delta \{u, v\}\in \mathcal{F}$.  Here $$X\Delta Y:=(X\cup Y)-(X\cap Y)$$  is the usual symmetric difference of sets.
	
	Let $D=(E, \mathcal{F})$ be a set system. An element $e\in E$ contained in every feasible set of $D$ is said to be a \emph{coloop}, while an element $e\in E$ contained in no feasible set of $D$ is said to be a \emph{loop}. For $A\subseteq E$, the \emph{twist} of $D$ with respect to $A$, denoted by $D*A$, is given by $$(E, \{A\Delta X: X\in \mathcal{F}\}).$$
	
	Let $D=(E, \mathcal{F})$ be a set system and $e\in E$. Then $D$ \emph{delete} by $e$, denoted $D\backslash e$, is defined as $D\backslash e:=(E\backslash e, \mathcal{F}')$, where
	\[\mathcal{F}':=\left\{\begin{array}{ll}
		\{F: F\in \mathcal{F}, F\subseteq E\backslash e\}, & \text{if $e$ is not a coloop,}\\
		
		\{F\backslash e: F\in \mathcal{F}\}, & \text{if $e$ is a coloop}.
	\end{array}\right.\]
	$D$ \emph{contract} by $e$, denoted $D/e$, is defined as $D/e:=(D*e)\backslash e$.
	
	For a finite set $E$, let $C$ be a symmetric $|E|$ by $|E|$ matrix over $GF(2)$, with rows and columns indexed, in the same order, by the elements of $E$. Let $C[A]$ be the principal submatrix of $C$ induced by the set $A\subseteq E$. We define the set system $D(C)=(E, \mathcal{F})$ with $$\mathcal{F}:=\{A\subseteq E: C[A] \mbox{ is non-singular}\}.$$  By convention $C[\emptyset]$ is non-singular. Then $D(C)$ is a delta-matroid \cite{AB4}. A delta-matroid is said to be \emph{binary} if it has a twist that is isomorphic to $D(C)$ for some symmetric matrix $C$ over $GF(2)$.
	
	For a set system $D=(E, \mathcal{F})$, let $r(D_{max})$  and $r(D_{min})$ denote the sizes of largest and  smallest feasible sets of $D$, respectively. The  \emph{width} of $D$, denote by $w(D)$, is defined by $$w(D):=r(D_{max})-r(D_{min}).$$

The twist polynomial introduced in \cite{Yan} for delta-matroids in fact can be defined for set systems.	
	
	\begin{definition}\cite{Yan}
		The twist polynomial $\partial_{\omega_{D}}(z)$ of a set system $D = (E, \mathcal{F})$ encodes the widths of
		its all twists, and is defined as follows:	
		\[\partial_{\omega_{D}}(z)=\sum_{A\subseteq E}z^{\omega(D*A)}.\]
	\end{definition}

	Let $D=(E, \mathcal{F})$ be a set system, and $a, b\in E$ with $a\neq b$.  One can define both the first Vassiliev move (handle sliding) and the second Vassiliev move (exchange handle ends) for set systems as follows.

	\begin{definition}\cite{Moffatt}\label{def01}
		The result of handle sliding of the element $a$ over the element $b$ is the set system $\widetilde{D}_{ab}=(E, \widetilde{\mathcal{F}}_{ab})$, where $$\widetilde{\mathcal{F}}_{ab}=\mathcal{F}\triangle \{F\cup a| F\cup b\in \mathcal{F} ~and~ F\subseteq  E\backslash  \{a,b\}\}.$$
	\end{definition}
	
	Note that $\widetilde{D}_{ab}$ may not be equal to $\widetilde{D}_{ba}$ since the set $\{F\cup a| F\cup b\in \mathcal{F} ~and~ F\subseteq  E\backslash  \{a,b\}\}$  may not be equal to the set $\{F\cup b| F\cup a\in \mathcal{F} ~and~ F\subseteq  E\backslash  \{b,a\}\}$.

	\begin{definition}\cite{Lando}\label{def02}
		The result of exchanging handle ends of the element $a$ and the element $b$ is the set system ${D'}_{ab}=(E, {\mathcal{F'}}_{ab})$, where $${\mathcal{F'}}_{ab}=\mathcal{F}\triangle \{F\cup \{a,b\}| F\in \mathcal{F} ~and~ F\subseteq  E\backslash  \{a,b\}\}.$$
	\end{definition}

	Since $\{F\cup \{a,b\}| F\in \mathcal{F} ~and~ F\subseteq  E\backslash  \{a,b\}\}=\{F\cup \{b,a\}| F\in \mathcal{F} ~and~ F\subseteq  E\backslash  \{b,a\}\},$ it follows that ${D'}_{ab}={D'}_{ba}$.

	\begin{proposition}\cite{Lando}\label{pro1} The handle sliding and the exchange handle ends for set systems possess the following properties.
		\begin{enumerate}
			\item [(1)] The handle sliding is an involution, $\widetilde{(\widetilde{D}_{ab})}_{ab}=D$;
			\item [(2)] The exchange handle ends is an involution, $(D'_{ab})'_{ab}=D$;
			\item [(3)] The handle sliding and the exchange handle ends commute, $(\widetilde{D}_{ab})'_{ab}=\widetilde{(D'_{ab})}_{ab}$. 
		\end{enumerate}
		
	\end{proposition}	
	
	Based on Definitions \ref{def01} and \ref{def02}, we can give a detailed explanation of $(\widetilde{D}_{ab})'_{ab}=\widetilde{(D'_{ab})}_{ab}$ (abbreviated as $\widetilde{D}'_{ab}$ and $\widetilde{D'}_{ab}$, respectively) as follows.
	
	\begin{proposition}\label{commute}
		The result of first exchanging handle ends and then sliding handle of the element $a$ over the element $b$ is the set system ${\widetilde{D'}}_{ab}=(E, {\mathcal{\widetilde{F'}}}_{ab})$, where $${\mathcal{\widetilde{F'}}}_{ab}=\mathcal{F}\triangle \{F\cup \{a,b\}| F\in \mathcal{F} ~and~ F\subseteq  E\backslash  \{a,b\}\}\triangle \{F\cup \{a\}| F\cup \{b\}\in \mathcal{F} ~and~ F\subseteq  E\backslash  \{a,b\}\}.$$
		Similarly, the result of first sliding handle and then exchanging handle ends of the element $a$ and the element $b$ is the set system ${\widetilde{D}'}_{ab}={\widetilde{D'}}_{ab}$.
	\end{proposition}
	
	\begin{proof}
		It is sufficient to prove that	
		\begin{eqnarray*}
			& &\mathcal{F'}_{ab}\triangle \{F\cup \{a\}| F\cup \{b\}\in \mathcal{F'}_{ab} ~and~ F\subseteq  E\backslash  \{a,b\}\}\\
			&=&\mathcal{F'}_{ab}\triangle \{F\cup \{a\}| F\cup \{b\}\in \mathcal{F} ~and~ F\subseteq  E\backslash  \{a,b\}\}.
		\end{eqnarray*}
		
		For any $\widehat{F}\in 	\mathcal{F'}_{ab}\backslash \{F\cup \{a\}| F\cup \{b\}\in \mathcal{F'}_{ab} ~and~ F\subseteq  E\backslash  \{a,b\}\}$,

		\begin{enumerate}
			
			\item [(1)] if $a\in \widehat{F}$ and $b\notin \widehat{F}$ and $(\widehat{F}\backslash a) \cup b\in  \mathcal{F'}_{ab}$, then $\widehat{F}\in \{F\cup \{a\}| F\cup \{b\}\in \mathcal{F'}_{ab} ~and~ F\subseteq  E\backslash  \{a,b\}\}$, a contradiction.
			\item [(2)] if $a\notin \widehat{F}$, then $\widehat{F}\notin \{F\cup \{a\}| F\cup \{b\}\in \mathcal{F} ~and~ F\subseteq  E\backslash  \{a,b\}\}$.
			\item [(3)] if $b\in \widehat{F}$, then $\widehat{F}\notin \{F\cup \{a\}| F\cup \{b\}\in \mathcal{F} ~and~ F\subseteq  E\backslash  \{a,b\}\}$.
			\item [(4)] if $a\in \widehat{F}$ and $b\notin \widehat{F}$ and $(\widehat{F}\backslash a) \cup b\notin  \mathcal{F'}_{ab}$, then $(\widehat{F}\backslash a) \cup b\notin \mathcal{F}$. Hence $\widehat{F}\notin \{F\cup \{a\}| F\cup \{b\}\in \mathcal{F} ~and~ F\subseteq  E\backslash  \{a,b\}\}$.
		\end{enumerate}
		Therefore $\widehat{F}\in \mathcal{F'}_{ab}\triangle \{F\cup \{a\}| F\cup \{b\}\in \mathcal{F} ~and~ F\subseteq  E\backslash  \{a,b\}\}$ for items (2)-(4).
		
		For any $\widehat{F}\in 	\{F\cup \{a\}| F\cup \{b\}\in \mathcal{F'}_{ab} ~and~ F\subseteq  E\backslash  \{a,b\}\}\backslash\mathcal{F'}_{ab}$, then $a\in \widehat{F}$, $b\notin \widehat{F}$ and  $\widehat{F}\backslash \{a\}\cup b\in \mathcal{F'}_{ab}=\mathcal{F}\triangle \{F\cup \{a,b\}| F\in \mathcal{F} ~and~ F\subseteq  E\backslash  \{a,b\}\}$. Hence $\widehat{F}\backslash \{a\}\cup b\in  \mathcal{F}$. Thus $\widehat{F}\in \{F\cup \{a\}| F\cup \{b\}\in \mathcal{F} ~and~ F\subseteq  E\backslash  \{a,b\}\}$. Therefore $\widehat{F}\in \mathcal{F'}_{ab}\triangle \{F\cup \{a\}| F\cup \{b\}\in \mathcal{F} ~and~ F\subseteq  E\backslash  \{a,b\}\}$.
		
		Consequently,
		\begin{eqnarray*}
			& &\mathcal{F'}_{ab}\triangle \{F\cup \{a\}| F\cup \{b\}\in \mathcal{F'}_{ab} ~and~ F\subseteq  E\backslash  \{a,b\}\}\\
			&\subseteq&\mathcal{F'}_{ab}\triangle \{F\cup \{a\}| F\cup \{b\}\in \mathcal{F} ~and~ F\subseteq  E\backslash  \{a,b\}\}.
		\end{eqnarray*}
		
		Similarly, we can prove that 
		\begin{eqnarray*}
			&&\mathcal{F'}_{ab}\triangle \{F\cup \{a\}| F\cup \{b\}\in \mathcal{F} ~and~ F\subseteq  E\backslash  \{a,b\}\}\\
			&\subseteq&\mathcal{F'}_{ab}\triangle \{F\cup \{a\}| F\cup \{b\}\in \mathcal{F'}_{ab} ~and~ F\subseteq  E\backslash  \{a,b\}\}.
		\end{eqnarray*}
		
		It follows that $${\mathcal{\widetilde{F'}}}_{ab}=\mathcal{F}\triangle \{F\cup \{a,b\}| F\in \mathcal{F} ~and~ F\subseteq  E\backslash  \{a,b\}\}\triangle \{F\cup \{a\}| F\cup \{b\}\in \mathcal{F} ~and~ F\subseteq  E\backslash  \{a,b\}\}.$$
		
		By a similar argument, we can prove that $${\mathcal{\widetilde{F}'}}_{ab}=\mathcal{F}\triangle \{F\cup \{a\}| F\cup \{b\}\in \mathcal{F} ~and~ F\subseteq  E\backslash  \{a,b\}\}\triangle \{F\cup \{a,b\}| F\in \mathcal{F} ~and~ F\subseteq  E\backslash  \{a,b\}\}.$$	
		
		Hence	${\widetilde{D}'}_{ab}={\widetilde{D'}}_{ab}$.
	\end{proof}

	\section{Twist polynomial as a weight system}
	
	First, we observe a close relationship between $ \mathcal{F}(D)$ and $\mathcal{F}(\widetilde{D}_{ab})$ as follows.
	
	\begin{lemma}\label{relation1}
		If	$\widehat{F}\in \mathcal{F}(D)\Delta \mathcal{F}(\widetilde{D}_{ab})$, then $a\in\widehat{F}$, $b\notin\widehat{F}$ and  $(\widehat{F}\backslash a) \cup b\in  \mathcal{F}(D)\cap \mathcal{F}(\widetilde{D}_{ab})$.
	\end{lemma}
	
	\begin{proof}
		Since
		\begin{eqnarray*}
			\widehat{F}\in \mathcal{F}(D)\Delta \mathcal{F}(\widetilde{D}_{ab})&=&
			\mathcal{F}(D)\Delta (\mathcal{F}(D)\Delta \{F\cup \{a\}| F\cup \{b\}\in \mathcal{F}(D) ~and~ F\subseteq  E\backslash  \{a,b\}\})\\
			&=&\{F\cup \{a\}| F\cup \{b\}\in \mathcal{F}(D) ~and~ F\subseteq  E\backslash  \{a,b\}\},
		\end{eqnarray*}
		it follows that $a\in \widehat{F}$, $b\notin \widehat{F}$ and $(\widehat{F}\backslash a) \cup b\in  \mathcal{F}(D)$. Since $$(\widehat{F}\backslash a) \cup b\notin  \{F\cup \{a\}| F\cup \{b\}\in \mathcal{F}(D) ~and~ F\subseteq  E\backslash  \{a,b\}\},$$ we have $(\widehat{F}\backslash a) \cup b\in \mathcal{F}(\widetilde{D}_{ab})$. Thus $(\widehat{F}\backslash a) \cup b\in  \mathcal{F}(D)\cap \mathcal{F}(\widetilde{D}_{ab})$.	
	\end{proof}
	
The following proposition is based on the relationship between $\mathcal{F}(D)$ and $\mathcal{F}(\widetilde{D}_{ab})$ in Lemma \ref{relation1}.
	
	\begin{proposition}\label{1stpair}
		Let $A\subseteq E$ and $|A\cap\{a, b\}|=0~\text{or}~2$. Then $\omega(D\ast A)=\omega(\widetilde{D}_{ab}\ast A)$.
	\end{proposition}
	
	\begin{proof}
		
		If there exists $F' \in \mathcal{F}(\widetilde{D}_{ab})\backslash \mathcal{F}(D)$ such that $r((\widetilde{D}_{ab}\ast A)_{min})=|F'\triangle A|$. Then $a\in F'$, $b\notin F'$ and $(F'\backslash a) \cup b\in \mathcal{F}(D)\cap \mathcal{F}(\widetilde{D}_{ab})$ by Lemma \ref{relation1}.
		\begin{description}
			\item[Case 1.] If $A\subseteq E\backslash  \{a,b\}$, then
			\[|(F'\backslash a) \cup b \triangle A|
			=|(F'\backslash a)  \triangle A|+1=|F'\triangle A|.\]
			\item[Case 2.] If $A\subseteq E$ and $a, b\in A$, then
			\[|(F'\backslash a) \cup b \triangle A|=|(F'\backslash a)  \triangle A|-1=(|F'\triangle A|+1)-1=|F'\triangle A|.\]
		\end{description}
		It follows that $$r((D\ast A)_{min})\leq |(F'\backslash a) \cup b \triangle A|=|F'\triangle A|=r((\widetilde{D}_{ab}\ast A)_{min}).$$
		Otherwise, there exists $F'\in \mathcal{F}(\widetilde{D}_{ab})\cap \mathcal{F}(D)$ such that $r((\widetilde{D}_{ab}\ast A)_{min})=|F'\triangle A|$. Then $$r((D\ast A)_{min})\leq |F'\triangle A|= r((\widetilde{D}_{ab}\ast A)_{min}).$$
		
		Since the handle sliding operation is involution, we have that $r((\widetilde{D}_{ab}\ast A)_{min})\leq r((D\ast A)_{min}).$
		Thus $$r((D\ast A)_{min})= r((\widetilde{D}_{ab}\ast A)_{min}).$$ By a similar argument, we can obtain that $$r((D\ast A)_{max})= r((\widetilde{D}_{ab}\ast A)_{max}).$$ Therefore  $$\omega(D\ast A)=\omega(\widetilde{D}_{ab}\ast A).$$
	\end{proof}
	
	\begin{remark}
		If $A=\emptyset$, then $\omega(D)=\omega(\widetilde{D}_{ab})$ by Proposition \ref{1stpair}, that is, the handle sliding operation does not change the width. But the exchange handle ends operation may change the width. For example, let $D=(\{a, b\}, \{\emptyset\})$. Then $D'_{ab}=(\{a, b\}, \{\emptyset, \{a, b\} \})$, but $\omega(D)\neq \omega(D'_{ab})$.
	\end{remark}
	
	Analysis similar to that in Lemma \ref{relation1} shows
	that there is a relationship between $\mathcal{F}(D)$ and $\mathcal{F}(D'_{ab})$ as follows.
	
	\begin{lemma}\label{relation2}
	
		If $\widehat{F}\in \mathcal{F}(D)\Delta\mathcal{F}(D'_{ab})$, then $a, b\in \widehat{F}$ and $\widehat{F}\backslash \{a,b\}\in \mathcal{F}(D)\cap \mathcal{F}(D'_{ab})$.
	\end{lemma}
	
	\begin{proof}
		Since
		\begin{eqnarray*}
			\widehat{F}\in \mathcal{F}(D)\Delta \mathcal{F}(D'_{ab})&=&
			\mathcal{F}(D)\Delta (\mathcal{F}(D)\Delta \{F\cup \{a,b\}| F\in \mathcal{F}(D) ~and~ F\subseteq  E\backslash  \{a,b\}\})\\
			&=&\{F\cup \{a,b\}| F\in \mathcal{F}(D) ~and~ F\subseteq  E\backslash  \{a,b\}\},
		\end{eqnarray*}
		we have $a, b\in \widehat{F}$ and $\widehat{F}\backslash \{a, b\}\in  \mathcal{F}(D)$. Since $$\widehat{F}\backslash \{a,b\}\notin\{F\cup \{a,b\}| F\in \mathcal{F}(D) ~and~ F\subseteq  E\backslash  \{a,b\}\},$$ it follows that $\widehat{F}\backslash \{a,b\}\in  \mathcal{F}(D'_{ab})$. Therefore $\widehat{F}\backslash \{a,b\}\in  \mathcal{F}(D)\cap \mathcal{F}(D'_{ab})$.
	\end{proof}
	
	Based on the relationship between $\mathcal{F}(D)$ and $\mathcal{F}(D'_{ab})$ as shown in Lemma \ref{relation2}, we derive  a statement similar to Proposition \ref{1stpair}.

	\begin{proposition}\label{2ndpair}
		Let $A\subseteq E$ and $|A\cap\{a, b\}|=1$. Then $\omega(D\ast A)=\omega(D'_{ab}\ast A)$.
	\end{proposition}
	
	\begin{proof}
		
		If there exists $F' \in \mathcal{F}(D'_{ab})\backslash \mathcal{F}(D)$ such that $r((D'_{ab}\ast A)_{min})=|F'\triangle A|$. Then $a, b\in F'$ and $F'\backslash \{a,b\} \in \mathcal{F}(D)\cap \mathcal{F}(D'_{ab})$ by Lemma \ref{relation2}. Therefore
		\[|(F'\backslash \{a,b\}) \triangle A|=|F'\triangle A|.\]
		Hence $$r((D\ast A)_{min})\leq|(F'\backslash \{a,b\}) \triangle A|=|F'\triangle A|= r((D'_{ab}\ast A)_{min}).$$
		Otherwise, there exists $F'\in \mathcal{F}(D'_{ab})\cap \mathcal{F}(D)$ such that $r((D'_{ab}\ast A)_{min})=|F'\triangle A|$.
		Then $$r((D\ast A)_{min})\leq |F'\triangle A|=r((D'_{ab}\ast A)_{min}).$$
		
		Since the exchange handle ends operation is involution, we have that $r((D'_{ab}\ast A)_{min})\leq r((D\ast A)_{min})$. Thus $$r((D'_{ab}\ast A)_{min})= r((D\ast A)_{min}).$$ By a similar argument, we show that $$r((D'_{ab}\ast A)_{max})= r((D\ast A)_{max}).$$ Therefore $$\omega(D'_{ab}\ast A)= \omega(D\ast A).$$
	\end{proof}
	
	\begin{definition}\cite{Lando}
		We say that an invariant $f$ of set systems satisfies the {\it four-term relation} if for any set system $D$ and a pair of distinct elements $a$ and $b$ in its ground set we have
		\begin{equation}\label{4term}
			f(D)+f(\widetilde{D'}_{ab})-f(D'_{ab})-f(\widetilde{D}_{ab})=0
		\end{equation}
	\end{definition}

Now we are in a position to state our main theorem.	
	
	\begin{theorem}\label{main}
		The twist polynomial of a set system satisfies the four-term relation.
	\end{theorem}
	
	\begin{proof}
		Let $D=(E,\mathcal{F})$ be a set system, and $a, b\in E$ with $a\neq b$.
		Let $D_1=D$, $D_2=D'_{ab}$, $D_3=\widetilde{D'}_{ab}$ and $D_4=\widetilde{D}_{ab}$.
		Then, it suffices to prove that \[\partial_{\omega_{D_{1}}}(z)+\partial_{\omega_{D_{3}}}(z)-\partial_{\omega_{D_{2}}}(z)-\partial_{\omega_{D_{4}}}(z)=0.\]
		For $i\in \{1, 2, 3, 4\}$, let $$D_{i;00}=\sum_{A\subseteq E\backslash  \{a,b\}}z^{\omega(D_{i}*A)},~~~~~~~
		D_{i;10}=\sum_{A\subseteq E\backslash  \{b\},~a\in A}z^{\omega(D_{i}*A)},
		$$
		$$D_{i;11}=\sum_{A\subseteq E,~\{a,b\}\in A}z^{\omega(D_{i}*A)},~~~~~~~
		D_{i;01}=\sum_{A\subseteq E\backslash  \{a\},~b\in A}z^{\omega(D_{i}*A)}.$$
	
		Then $$\partial_{\omega_{D_i}}(z)=D_{i;00}+D_{i;10}+D_{i;01}+D_{i;11}.$$
		By Proposition \ref{1stpair}, we obtain
		\[D_{1;00}=\sum_{A\subseteq E\backslash  \{a,b\}}z^{\omega(D_{1}*A)}=\sum_{A\subseteq E\backslash  \{a,b\}}z^{\omega(D_{4}*A)}= D_{4;00}\]
		and
		\[D_{1;11}=\sum_{A\subseteq E,~\{a,b\}\in A}z^{\omega(D_{1}*A)}=\sum_{A\subseteq E,~\{a,b\}\in A}z^{\omega(D_{4}*A)}= D_{4;11}.\]
		By $D_3=(\widetilde{D_{2}})_{ab}$ and Proposition \ref{1stpair}, we have
		\[D_{2;00}=\sum_{A\subseteq E\backslash  \{a,b\}}z^{\omega(D_{2}*A)}=\sum_{A\subseteq E\backslash  \{a,b\}}z^{\omega(D_{3}*A)}= D_{3;00}\]
		and
		\[D_{2;11}=\sum_{A\subseteq E,~\{a,b\}\in A}z^{\omega(D_{2}*A)}=\sum_{A\subseteq E,~\{a,b\}\in A}z^{\omega(D_{3}*A)}= D_{3;11}.\] 	
		By Proposition \ref{2ndpair}, it follows that
		\[D_{1;10}=\sum_{A\subseteq E\backslash  \{b\},~a\in A}z^{\omega(D_{1}*A)}=\sum_{A\subseteq E\backslash  \{b\},~a\in A}z^{\omega(D_{2}*A)}= D_{2;10}\]
		and
		\[D_{1;01}=\sum_{A\subseteq E\backslash  \{a\}, ~b\in A}z^{\omega(D_{1}*A)}=\sum_{A\subseteq E\backslash  \{a\},~b\in A}z^{\omega(D_{2}*A)}= D_{2;01}.\]
		By $D_3=(D_{4}')_{ab}$ and Proposition \ref{2ndpair}, we see that
		\[D_{4;10}=\sum_{A\subseteq E\backslash  \{b\},~a\in A}z^{\omega(D_{4}*A)}=\sum_{A\subseteq E\backslash  \{b\},~a\in A}z^{\omega(D_{3}*A)}= D_{3;10}\]
		and
		\[D_{4;01}=\sum_{A\subseteq E\backslash  \{a\},~b\in A}z^{\omega(D_{4}*A)}=\sum_{A\subseteq E\backslash  \{a\},~b\in A}z^{\omega(D_{3}*A)}= D_{3;01}.\]	
		Consequently, we can verify that 	$$\partial_{\omega_{D_{1}}}(z)+\partial_{\omega_{D_{3}}}(z)-\partial_{\omega_{D_{2}}}(z)-\partial_{\omega_{D_{4}}}(z)=0.$$
	\end{proof}
	
	\begin{example}
		Let $D=(\{a,b,c\}, \{\emptyset, \{b,c\}\})$. Then
		$${D'}_{ab}=(\{a,b,c\}, \{\emptyset,\{a,b\} , \{b,c\}\}),$$ $$\widetilde{D}_{ab}=(\{a,b,c\}, \{\emptyset,\{a,c\} , \{b,c\}\}),$$
		$$\widetilde{D'}_{ab}=(\{a,b,c\}, \{\emptyset,\{a,b\},\{a,c\},\{b,c\}\}).$$
		Moreover, we can calculate their twist polynomials as follows.
		\[\partial_{\omega_{{D}}}(z)=4z^2+4,~ ~  \partial_{\omega_{{\widetilde{D}_{ab}}}}(z)=6z^2+2,~~ \partial_{\omega_{{{D'}_{ab}}}}(z)=6z^2+2,~~ \partial_{\omega_{{\widetilde{D'}}_{ab}}}(z)=8z^2.\]
		Then we can verify that $$\partial_{\omega_{D}}(z)+\partial_{\omega_{\widetilde{D'}_{ab}}}(z)-\partial_{\omega_{D'_{ab}}}(z)-\partial_{\omega_{\widetilde{D}_{ab}}}(z)=0.$$
		Thus the twist polynomial of $D$ on $\{a,b\}$ satisfies four-term relation.
		The other cases can be verified similarly.
	\end{example}

	\section{Applied to delta-matroids and ribbon graphs}
	
	Moffatt and Mphako-Banda \cite{Moffatt} showed that a handle sliding applied to a delta-matroid can produce a set system that is not a delta-matroid. But both the first and the second Vassiliev moves preserve the class of binary delta-matroids as follows.
	
	\begin{proposition}\cite{Lando, Moffatt}\label{pro13}
		If $D=(E, \mathcal{F})$ is a binary delta-matroid, and $a, b\in E$ with $a\neq b$, then $\widetilde{D}_{ab}$ and $D'_{ab}$ are binary delta-matroids.
	\end{proposition}
	
	Since the binary delta-matroids are special cases of set systems, we can deduce the following statement by Theorem \ref{main} and Proposition \ref{pro13}.

	\begin{theorem}
		The twist polynomial of a binary delta-matroid satisfies the four-term relation.
	\end{theorem}

	Ribbon graphs are well-known to be equivalent to cellularly embedded graphs. The reader is referred to \cite{EM} for further details about ribbon graphs. A \emph{quasi-tree} is a ribbon graph with one boundary component.

	Let $G=(V, E)$ be a ribbon graph and let $$\mathcal{F}:=\{F\subseteq E(G): \text{$F$ is the edge set of a spanning quasi-tree of $G$}\}.$$ We call $D(G)=:(E, \mathcal{F})$ the \emph{delta-matroid} \cite{CMNR} of $G$.
	
	\begin{lemma}\cite{Yan}\label{GD}
		Let $G = (V, E)$ be a ribbon graph. Then $$\partial_{\omega_{D(G)}}(z)=\partial_{\varepsilon_{G}}(z),$$
		
\noindent	where $\partial_{\varepsilon_{G}}(z)$ is the partial-dual polynomial of $G$.
	\end{lemma}
	
	\begin{proposition}\cite{Lando,Moffatt}\label{slideab}
		Let $G = (V, E)$ be a ribbon graph, $a$ and $b$ be distinct edges of $G$ with neighbouring ends, let $\widetilde{G}_{ab}$ and $G'_{ab}$ be the ribbon graphs obtained from $G$ by handle sliding $a$ over $b$ and by exchange handle ends $a$ and $b$, respectively. Then
		\[D(\widetilde{G}_{ab})=\widetilde{D(G)}_{ab}.\]
		and \[D(G'_{ab})=D(G)'_{ab}.\]
	\end{proposition}
	
	\begin{theorem}\cite{Chmutov2023,Deng}\label{DDJY}\label{three4T}
		The partial-dual polynomial as a function on framed chord diagrams satisfies three four-term relations. In particular, the partial-dual polynomial as a function on bouquets (one-vertex ribbon graphs) satisfies the four-term relation. 

	\end{theorem}

Note that three four-term relations \cite{Deng} in Theorem \ref{three4T} can be unified as the following equation
			\begin{equation}\label{ribbon}
		\partial_{\varepsilon_{G}}+\partial_{\varepsilon_{\widetilde{G'}_{ab}}}-\partial_{\varepsilon_{G'_{ab}}}-\partial_{\varepsilon_{\widetilde{G}_{ab}}}=0.
	\end{equation}
	
	We now demonstrate that Theorem \ref{main} clearly implies Theorem \ref{DDJY}. Note that the following equation (\ref{deltamatroid}) can be verified by Proposition \ref{slideab} and Theorem \ref{main}. 
	And equation (\ref{deltamatroid}) implies equation (\ref{ribbon}) by Lemma \ref{GD}.
	\begin{equation}\label{deltamatroid} \partial_{\omega_{D(G)}}+\partial_{\omega_{D(\widetilde{G'}_{ab})}}-\partial_{\omega_{D(G'_{ab})}}-\partial_{\omega_{D(\widetilde{G}_{ab})}}=0.
	\end{equation}

	We say that an invariant $f$ of delta-matroids satisfies the Tutte relations \cite{Lando} if, for any delta-matroid $D=(E, \mathcal{F})$, we have
	\begin{eqnarray*}
		f(D)=\left\{\begin{array}{ll}
			xf(D \setminus e), & \mbox{if}~e~\mbox{is a loop of}~D\\
			y f(D / e), & \mbox{if}~e~\mbox{is a coloop of}~D\\
			pf(D \setminus e) + qf(D / e), & \mbox{otherwise}.
		\end{array}\right.
	\end{eqnarray*}
	Here $x, y, p$ and $q$ are some indeterminates.

	Lando and Zhukov \cite{Lando} proved that any invariant of binary delta-matroids satisfying the Tutte relation satisfies also the four-term relation.

	\begin{proposition}\cite{Lando}
		If $f$ is an invariant of binary delta-matroid satisfying the Tutte relation, then $f$ satisfies the four-term relation.
	\end{proposition}

	Let $D$ be a binary delta-matroid. Here is an examples to show that the twist polynomial of $D$ does not satisfy the Tutte relation.
	
	\begin{example}
		Let $E=\{a, b, c\}$,
		
		$$\mathcal{F}=\{\emptyset, \{b,c\}\},$$ and $$\mathcal{F}'= \{\emptyset,\{a,b\},\{a,c\},\{b,c\}\}.$$
		Then $D=(E, \mathcal{F})$ and $D'=(E, \mathcal{F}')$ are both binary delta-matroids.
		Suppose that the twist polynomial of binary delta-matroids satisfying the Tutte relation. We have \begin{equation}\label{equ2}
			\left\{
			\begin{aligned}
				4z^2+4=\partial_{\omega_{{D}}}(z)&=x\partial_{\omega_{{D-b}}}(z)+y\partial_{\omega_{{D/ b}}}(z)=4(x+y),\\
				8z^2=\partial_{\omega_{{D'}}}(z)&= x\partial_{\omega_{{D'-c}}}(z)+y\partial_{\omega_{{D'/ c}}}(z)=(x+y)(2z^2+2).\\
			\end{aligned}
			\right.
		\end{equation}
	So $32z^2=(4z^2+4)(2z^2+2)$, a contradiction.
		
	\end{example}
	
	\begin{remark}
		Note that we find an invariant (twist polynomial) of binary delta-matroids that does not satisfy the Tutte relation, but satisfies the four-term relation. Consequently, it produces a new concrete weight system.
	\end{remark}

	\section{Acknowledgements}
	This work is supported by NSFC (Nos. 12171402, 12101600), and partially supported by the Natural Science Foundation of Hunan Province, PR China (No. 2022JJ40418), the Excellent Youth Project of Hunan Provincial Department of Education, PR China (No. 23B0117), and the China Scholarships Council (Grant No. 202108430063). The first author thanks the National Institute of Education, Nanyang Technological University, where part of this research was performed.

\end{document}